\documentclass[oneside]{amsart}
\usepackage[utf8]{inputenc}
\usepackage{hyperref}
\usepackage[OkabeIto]{colorblind}

\title{Tukey morphisms between finite relations}
\author{Rhett Barton}
\address{Rhett Barton, Boise State University and Blue Cross of Idaho}
\email{rhett.bart@gmail.com}
\author{Samuel Coskey}
\address{Samuel Coskey, Boise State University and UCL}
\email{scoskey@gmail.com}
\author{Paul Ellis}
\address{Paul Ellis, Manhattanville University and Rutgers University}
\email{paulellis@paulellis.org}


\usepackage{geometry}
\usepackage{setspace}
\usepackage{mathpazo,bm,amssymb}
\usepackage[textwidth=30mm]{todonotes}
\usepackage{comment}
\usepackage{hyperref}

\usepackage{tikz}
\usetikzlibrary{decorations.pathmorphing}

\newtheorem{theorem}{Theorem}[section]
\newtheorem{lemma}[theorem]{Lemma}
\newtheorem{proposition}[theorem]{Proposition}
\newtheorem{corollary}[theorem]{Corollary}
\theoremstyle{definition}
\newtheorem{definition}[theorem]{Definition}
\theoremstyle{remark}

\newtheorem{question}{Question}

\newcommand{\id}{\mathrm{id}}

\usepackage{etoolbox}
\makeatletter\pretocmd{\@seccntformat}{\S}{}{}
  \pretocmd{\@subseccntformat}{\S}{}{}\makeatother

\begin{document}

\begin{abstract}
  We investigate Tukey morphisms between binary relations, establishing several fundamental lemmas.  We then specialize to finite binary relations, using computational methods to classify all binary relations with at most $6$ points in the domain and codomain up to bimorphism.  Finally we give a construction of finite binary relations with arbitrary dominating number and dual dominating number.
\end{abstract}

\maketitle

\section{Introduction}
\label{sec:intro}

A \emph{binary relation} is a triple $\bm{A}=(A_-,A_+,A)$, where $A_\pm$ are arbitrary sets and $A\subseteq A_-\times A_+$. Given binary relations $\bm{A},\bm{B}$, a \emph{Tukey morphism} or simply \emph{morphism} from $\bm{A}$ to $\bm{B}$ is a pair of functions $\phi=(\phi_-,\phi_+)$ where
\begin{align*}
  \phi_-&\colon B_-\to A_-\\
  \phi_+&\colon A_+\to B_+\text{,}
\end{align*}
and for all $b\in B_-$ and $a\in A_+$,
\begin{equation}
  \label{eq:morphism}
  \phi_-(b)\mathrel{A}a\implies b\mathrel{B}\phi_+(a)\text{.}
\end{equation}

As a mnemonic for the condition \eqref{eq:morphism}, we draw the generalized commuting diagram shown below, where the wavy lines represent relations rather than functions.
\begin{center}
  \begin{tikzpicture}[scale=2]
    \node at (0, 0) (Ap)  {$A_+$};
    \node at (1, 0) (Bp) {$B_+$};
    \draw[->] (Ap) -- node[above] {$\phi_+$} (Bp);
    \node at (0,-1) (Am)  {$A_-$}
      edge[->,decorate,decoration={snake,amplitude=.5}]
      node[left]{$A$} (Ap);
    \node at (1,-1) (Bm) {$B_-$}
      edge[->,decorate,decoration={snake,amplitude=.5}]
      node[right]{$B$} (Bp);
    \draw[->] (Bm) -- node[below] {$\phi_-$} (Am);
  \end{tikzpicture}
\end{center}

We use the notation $\phi\colon\bm{A}\to\bm{B}$ to express that $\phi$ is a morphism from $\bm{A}$ to $\bm{B}$, and the notation $\bm{A}\to\bm{B}$ to express that there exists a morphism from $\bm{A}$ to $\bm{B}$. We say that $\bm{A}$ and $\bm{B}$ are \emph{bimorphic} if $\bm{A}\to\bm{B}$ and $\bm{B}\to\bm{A}$.  See Figure \ref{fig:example of a morphism} for an example of a morphism.

\begin{figure}[!h]
  \centering
  \begin{tikzpicture}[scale=.5]
    \draw[fill=black] (0,0) circle (2pt);
    \draw[fill=black] (2,0) circle (2pt);
    \draw[fill=black] (4,0) circle (2pt);
    \draw[fill=black] (0,-4) circle (2pt);
    \draw[fill=black] (2,-4) circle (2pt);
    \draw[fill=black] (4,-4) circle (2pt);
    \draw[thick] (0,0) -- (0,-4);
    \draw[thick] (2,0) -- (2,-4);
    \draw[thick] (4,0) -- (4,-4);
    \draw[thick] (4,0) -- (2,-4);
    \node at (-1,0) {$A_{+}$};
    \node at (-1,-4) {$A_{-}$};
    \node at (-1,-2) {$A$};
    \draw[fill=black] (0+11,0) circle (2pt);
    \draw[fill=black] (2+11,0) circle (2pt);
    \draw[fill=black] (4+11,0) circle (2pt);
    \draw[fill=black] (0+11,-4) circle (2pt);
    \draw[fill=black] (2+11,-4) circle (2pt);
    \draw[fill=black] (4+11,-4) circle (2pt);
    \draw[thick] (0+11,0) -- (0+11,-4);
    \draw[thick] (2+11,0) -- (2+11,-4);
    \draw[thick] (4+11,0) -- (4+11,-4);
    \draw[thick] (0+11,0) -- (2+11,-4);
    \node at (-1+11,0) {$B_{+}$};
    \node at (-1+11,-4) {$B_{-}$};
    \node at (-1+11,-2) {$B$};
    \draw[->,  thick, OI1] (2,0.2) to [bend left] (4+11,0.2);
    \draw[->,  thick, OI2] (0,0.2) to [bend left] (0+11,0.2);
    \draw[->,  thick, OI1] (4,0.2) to [bend left] (4+11,0.2);
    \draw[->, thick, OI2] (0+11,-4.2) to [bend left] (0,-4.2);
    \draw[->, thick, OI2] (2+11,-4.2) to [bend left] (0,-4.2);
    \draw[->, thick, OI1] (4+11,-4.2) to [bend left] (2,-4.2);
  \end{tikzpicture}
  \caption{A morphism from $\bm{A}$ to $\bm{B}$.\label{fig:example of a morphism}}
\end{figure}

Observe that Tukey morphisms may be composed, which justifies the use of the category theoretic term. Indeed, if $\phi\colon\bm{A}\to\bm{B}$ and $\psi\colon\bm{B}\to\bm{C}$, then letting $\psi\circ\phi=(\phi_-\circ\psi_-,\psi_+\circ\phi_+)$ we have $\psi\circ\phi\colon\bm{A}\to\bm{C}$. As a conseqence, morphisms define a quasi-order (reflexive and transitive) on binary relations, and bimorphism defines an equivalence relation on binary relations.

Tukey morphisms and several close relatives of Tukey morphisms arise in many areas of mathematics. They were introduced in the context of topology and filterbases, where they are closely related to cofinal mappings, see \cite{tukey}. Many other applications of Tukey morphisms have since emerged in topology and measure theory, see for instance \cite{todorcevic,solecki}. In algebra, Tukey morphisms appear in the context of generalized Galois connections \cite{erne}. In set theory, Tukey morphisms appear in the context of ultrafilters \cite{dobrinen}, as well as in the study of cardinal characteristics of the continuum \cite{blass}. In this context, \cite{blass} began to study Tukey morphisms for their own sake. He later introduced the notion of \emph{definable} Tukey morphism, see for instance \cite{blass-borel,mildenberger,coskey}.

In this article we study Tukey morphisms in their own right, and explore some of their fundamental properties. We furthermore begin to investigate the combinatorics of Tukey morphisms between \emph{finite} binary relations.

Before summarizing the results in the rest of the paper, we recall some background material concerning the notion of dominating families, dominating numbers, and dual relations.

\begin{definition}
  \label{def:dom number}
  Let $\bm{A}$ be a binary relation. A subset $Y\subseteq A_+$ is called an \emph{A-dominating family} if for every $a \in A_{-}$ there exists $\alpha \in Y$ such that $a\mathrel{A}\alpha$.  The \emph{dominating number} $\delta(\mathbf{A})$ of $\bm{A}$ is the minimum size of an $\bm{A}$-dominating family. If there is no A-dominating family, we set $\delta(\mathbf{A})=\infty$.
\end{definition}

It is clear that a dominating family of minimum size is also minimal, in the sense that it contains no proper subsets which are dominating families. We remark that there may be minimal dominating families which are not of minimum size, such as the example in Figure~\ref{fig:dom families}.
    
\begin{figure}[!h]
    \centering
    \begin{tikzpicture}[scale=.5]
        \draw[fill=black] (0,0) circle (2pt);
        \draw[fill=black] (2,0) circle (2pt);
        \draw[fill=black] (4,0) circle (2pt);
        \draw[fill=black] (6,0) circle (2pt);
        \draw[fill=black] (0,-4) circle (2pt);
        \draw[fill=black] (2,-4) circle (2pt);
        \draw[fill=black] (4,-4) circle (2pt);
        \draw[thick] (0,0) -- (0,-4);
        \draw[thick] (2,0) -- (2,-4);
        \draw[thick] (4,0) -- (4,-4);
        \draw[thick] (6,0) -- (0,-4);
        \draw[thick] (6,0) -- (2,-4);
        \draw[thick] (6,0) -- (4,-4);
        \node at (-1,-2) {$A$};
        \node at (-1,-4) {$A_{-}$};
        \node at (-1,0) {$A_{+}$};
        \draw[thick, red] (2,0) ellipse (2.25 and 0.5); 
        \draw[thick, blue] (6,0) circle (10pt); 
    \end{tikzpicture}
    \caption{The red ellipse denotes a minimal dominating family, the blue circle denotes a dominating family of minimum size.\label{fig:dom families}}
\end{figure}

The dominating number is a key invariant of a binary relation. The following result states that Tukey morphisms respect the dominating number order.

\begin{proposition}[Blass \cite{blass}]\
  \label{prop:blass}
  \begin{itemize}
    \item Let $\phi$ be a morphism from $\bm{A}$ to $\bm{B}$. If $\mathcal F$ is an $\bm{A}$-dominating family, then $\phi_+(\mathcal F)$ is a $\bm{B}$-dominating family.
    \item If $\bm{A}\to\bm{B}$ then $\delta(\bm{A})\geq\delta(\bm{B})$.
  \end{itemize}
\end{proposition}

The first item is a consequence of the implication \eqref{eq:morphism}, and the second item is an immediate consequence of the first item.

Binary relations together with Tukey morphisms admit a natural notion of duality.

\begin{definition}
  Given a binary relation $\mathbf{A} = (A_{-},A_{+},A)$, the \emph{dual} of $\mathbf{A}$ is the binary relation $\mathbf{A}^\bot = (A_{+},A_{-},\neg \Check{A})$, where $\neg\Check{A}$ denotes the complement of the converse of $A$, that is, $x\mathrel{\mathord{\neg}\Check{A}}y$ if and only if $y \not\mathrel{A} x$.  We call $\delta(\bm{A}^\perp)$ the \emph{dual dominating number} of $\bm{A}$.
\end{definition}

We observe that $(\bm{A}^\perp)^\perp=\bm{A}$. The next result states that Tukey morphisms interact with duals in the usual way.

\begin{proposition}[Blass \cite{blass}]\
  \label{prop:dual dominating number}
  \begin{itemize}
    \item If $(\phi_-,\phi_+)$ is a morphism from $\bm{A}$ to $\bm{B}$, then $(\phi_+,\phi_-)$ is a morphism from $\bm{B^\perp}$ to $\bm{A^\perp}$.
    \item We have $\bm{A}\to\bm{B}$ if and only if $\bm{B^\perp}\to\bm{A^\perp}$.
  \end{itemize}
\end{proposition}

The first item amounts to taking the contrapositive of the implication condition \eqref{eq:morphism}. The second item is an immediate consequence of the first item.

\begin{corollary}
  \label{cor:dual dom number}
  If $\bm{A}\to\bm{B}$ then $\delta(\bm{A})\geq\delta(\bm{B})$ and $\delta(\bm{B}^\perp)\geq\delta(\bm{A}^\perp)$
\end{corollary}

It is natural to ask whether the converse of this result holds, that is, if $\delta(\bm{A})\geq\delta(\bm{B})$ and $\delta(\bm{B}^\perp)\geq\delta(\bm{A}^\perp)$, does it follow that $\bm{A}\to\bm{B}$? In \cite{yiparaki}, Yiparaki gives counterexamples using infinite binary relations. In Section~\ref{sec:small}, we find counterexamples using finite binary relations.


The rest of this paper is organized as follows.
In Sections~\ref{sec:examples} and~\ref{sec:building}, we develop elementary results and explore general properties of binary relations and morphisms.
In Section~\ref{sec:lemmas}, we develop an algorithm which produces, from a finite binary relation, a somewhat canonical relation bimorphic to it, called its skeleton.
In Section \ref{sec:small}, we use a computer to calculate the skeletons of all binary relations $\bm{A}$ such that $|A_\pm|\leq6$, and to classify the binary relations such that $|A_\pm|\leq6$ up to bimorphism.
In Section \ref{sec:construction}, we briefly consider the structure of the morphism quasi-order on binary relations. In particular, we use insights from our classification to develop a method for constructing binary relations $\bm{A}$ with any given finite values of $\delta(\bm{A})$ and $\delta(\bm{A}^\perp)$.

While the focus of this paper is finite relations, we only use finiteness to ensure that the algorithm in Definition \ref{definition: skeleton} terminates.  Note that all the preceeding results work in the infinite case as well.

\textbf{Acknowledgement.} Many of the results in this paper were initially obtained by Rhett Barton while a graduate student at Boise State University, and appear in his master's thesis \cite{barton}.

\section{Some fundamental morphisms}
\label{sec:examples}

As a warm-up, we begin by addressing trivial cases when the dominating number is $1$ or $\infty$.

\begin{proposition} \;
  \label{prop: one infinity}
  \begin{itemize}
    \item If $\delta(\bm{B})=1$ and $\bm{A}$ is any binary relation, then $\bm{A}\to\bm{B}$.
    \item If $\delta(\bm{A})=\infty$ and $\bm{B}$ is any binary relation, then $\bm{A}\to\bm{B}$.
  \end{itemize}
\end{proposition}

For the first item, if $\delta(\bm{B})=1$ then there is some $y_0\in B^+$ satisfying $xBy_0$ for all $x\in B_-$. If $\bm{A}$ is any other relation, the map $\phi_+(a)=y_0$ will always define a morphism from $\bm{A}$ to $\bm{B}$. The second item follows from the first item, the fact that $\delta(\bm{A})=1\iff \delta(\bm{A}^\perp)=\infty$, and Proposition~\ref{prop:dual dominating number}. The second item, together with Proposition \ref{prop:blass}, explains why a relation without a dominating family is defined to have dominating number infinity.

From now on, unless otherwise specified, we will work only with binary relations $\bm{A}$ satisfying $1<\delta(\bm{A})<\infty$ (and thus $1<\delta(\bm{A}^\perp)<\infty$ as well).

In later analysis, we will have special use for the following key binary relations.

\begin{definition}
  A binary relation $\bm{A}=(A_-,A_+,A)$ is called an \emph{$n$-ladder} if $\lvert A_- \rvert=\lvert A_+ \rvert=n\geq1$ and $A$ is (the graph of) a bijective function from $A_-$ to $A_+$.
\end{definition}

If $\bm{A}$ is an $n$-ladder then $\delta(\bm{A})=n$; if additionally $n\geq2$ then $\delta(\bm{A}^\perp)=2$.

\begin{lemma}
  \label{lemma-ladders dominate}
  Let $\bm{A}$ be an $n$-ladder. If $\delta(\bm{B})\leq n$, then $\bm{A}\to\bm{B}$.
\end{lemma}

\begin{proof}
  Let $b_1,\ldots, b_k\in B_+$ be a $B$-dominating family with $k\leq n$.  Let $\phi_+\colon A_+\to B_+$ be function which is surjective onto $\{b_1,\ldots, b_k\}$.  Now consider $d\in B_-$.  Suppose $d\mathrel{B}b_i$ and $\phi_+(a)=b_i$ where $a\in A_+$.  Then there is a unique $c\in A_-$ such that $c\mathrel{A}a$. Define $\phi_-(d)=c$.  Since $a$ is the unique element of $A_+$ satisfying $c\mathrel{A}a$, the definition of morphism is satisfied.
\end{proof}

\begin{corollary}
  \label{cor: ladders}
  If $\bm{A}$ is an $n$-ladder, $\bm{B}$ is a $k$-ladder, and $k\leq n$, then $\bm{A}\to\bm{B}$.
\end{corollary}

Observe that the $2$-ladder is isomorphic to its own dual. This allows us to achieve the following result.

\begin{corollary}
  If $\delta(\bm{A})=\delta(\bm{A}^\perp)=2$, then $\bm{A}$ is bimorphic with the $2$-ladder.
\end{corollary}

\begin{proof}
  Suppose $\delta(\bm{A})=\delta(\bm{A}^\perp)=2$, and let $\bm{B}$ be a $2$-ladder.  By Lemma~\ref{lemma-ladders dominate}, $\bm{B}\to\bm{A}$.  Since $\bm{B}^\perp$ is also a $2$-ladder, we have $\bm{B}^\perp\to\bm{A}^\perp$.  So by Proposition~\ref{prop:dual dominating number}, $\bm{A}\to\bm{B}$.
\end{proof}

\begin{corollary}
  \label{cor:dual bimorphic}
  If $\delta(\bm{A})=\delta(\bm{A}^\perp)=2$, then $\bm{A}$ and $\bm{A}^\perp$ are bimorphic.
\end{corollary}

\section{Building morphisms}
\label{sec:building}

In this section we explore several situations where a morphism can be constructed from one binary relation to another. In the following lemma, a \emph{homomorphism} from $\bm{A}$ to $\bm{B}$ is a function $\psi\colon A_-\cup A_+\to B_-\cup B_+$ with the property that and $\psi(A_-)\subseteq B_-$, $\psi(A_+)\subseteq B_+$, and $a\mathrel{A}a'\iff \psi(a)\mathrel{B}\psi(a')$.

\begin{lemma}
  \label{lem:surjhomomorphism}
  Suppose $\psi$ is a homomorphism from $\bm{A}$ to $\bm{B}$ with the property that $\psi\vert_{A_-}$ is surjective onto $B_-$.  Then $\bm{A}\to\bm{B}$.
\end{lemma}

\begin{proof}
  Select a function $g\colon B_-\to A_-$ such that $\psi\vert_{A_-}\circ g=\id_{B_-}$.  Then it is not difficult to check that $(g,\psi\vert_{A_+})$ is a morphism from $\bm{A}$ to $\bm{B}$.
\end{proof}

We next examine some natural cases where making a small change to a binary relation results in a new binary relation such that there exists a morphism from the old to the new.

\begin{lemma}\label{lemma:identity morphisms}
  Let $\bm{A}$ and $\bm{B}$ be binary relations. Suppose that any of the following conditions holds:
  \begin{enumerate}
    \item $A_-=B_-$, $A_+=B_+$, and $A\subseteq B$
    \item $\bm{B}$ is an induced subrelation obtained by deleting points from $A_-$. That is, $\bm{B}=(S,A_+,B)$ where $S\subseteq A_-$ and $c\mathrel{B}a\iff c\mathrel{A}a$ for all $c\in S, a\in A_+$.
    \item $\bm{A}$ is an induced subrelation obtained by deleting points from $B_+$. That is, $\bm{A}=(B_-,S, A)$ where $S\subseteq B_+$ and $c\mathrel{A}a\iff c\mathrel{B}a$ for all $c\in B_-, a\in S$.
  \end{enumerate}
  Then $\bm{A}\to\bm{B}$.
\end{lemma}

\begin{proof}
  In each case the identity (or inclusion) maps $(\id_{B_-},\id_{A_+})$ together form a morphism. 
\end{proof}





In the rest of this section, we define and study the disjoint union of two binary relations.

\begin{definition}
  Let $\bm{A}$ and $\bm{B}$ be binary relations.  Then $\bm{A}\oplus\bm{B}=(A_-\sqcup B_-, A_+\sqcup B_+, A\oplus B)$, where $c(A\oplus B)d$ if and only if $c\mathrel{A}d$ or $c\mathrel{B}d$.
\end{definition}

The next two result show that morphisms behave reasonably well with respect to disjoint unions.

\begin{lemma}
  Let $\bm{A},\bm{A'},\bm{B},\bm{B'}$ be binary relations.  If $\bm{A}\to\bm{B}$ and $\bm{A'}\to\bm{B'}$, then $\bm{A}\oplus\bm{A'}\to\bm{B}\oplus\bm{B'}$.
\end{lemma}

\begin{proof}
  Suppose $\phi$ witnesses $\bm{A}\to\bm{B}$ and $\phi'$ witnesses $\bm{A'}\to\bm{B'}$.  Then $(\phi_-\sqcup\phi'_-,\phi_+\sqcup\phi'_+)$ witnesses $\bm{A}\oplus\bm{A'}\to\bm{B}\oplus\bm{B'}$.
\end{proof}

\begin{lemma}
  \label{lem-disjoint union converse}
  Suppose $\bm{C},\bm{B},\bm{B'}$ are relations whose underlying sets are all disjoint, and $\bm{C}\to\bm{B}\oplus\bm{B'}$.  Then there exist induced subrelations $\bm{A},\bm{A'}$ of $\bm{C}$ such that $C_-=A_-\sqcup A'_-$, $C_+=A_+\sqcup A'_+$, $\bm{A}\to\bm{B}$, and $\bm{A'}\to\bm{B'}$.
\end{lemma}

\begin{proof}
  Suppose $\phi$ is a morphism from $\bm{C}$ to $\bm{B}\oplus\bm{B'}$.  Define $A_-=\phi_-(B_-)$ and $A'_-=C_-\setminus A_-$.  Next define
  \[A_+=\{a\in C_+\mid \text{ there exists }c\in A_-\text{ such that }c\mathrel{C}a\}
  \]
  and $A'_+=C_+\setminus A_+$.
  
  It is easy to see that $(\phi_-\vert_{B_-},\phi_+\vert_{A_+})$ is a morphism from $\bm{A}$ to $\bm{B}$.

  To construct a morphism $\psi$ from $\bm{A'}$ to $\bm{B'}$, first set $\psi_-=\phi_-\vert_{B_-}$.  Next consider  $a\in A_+$.  If there is some $c\in\phi_-(B_-)$ with $cA'a$, then set $\psi_+(a)=\phi_+(a)$.  Otherwise assign $\psi_+(a)$ to any element of $B'_+$.
\end{proof}

In Lemma \ref{lem-disjoint union converse}, we cannot guarantee $\bm{C}=\bm{A}\oplus \bm{A'}$. See Figure~\ref{fig:morphism vs union} for a counterexample.

\begin{figure}[!h]
  \centering
  \begin{tikzpicture}[scale=.5]
    \draw[fill=black] (0,0) circle (2pt);
    \draw[fill=black] (2,0) circle (2pt);
    \draw[fill=black] (4,0) circle (2pt);
    \draw[fill=black] (6,0) circle (2pt);
    \draw[fill=black] (8,0) circle (2pt);
    \draw[fill=black] (0,-4) circle (2pt);
    \draw[fill=black] (2,-4) circle (2pt);
    \draw[fill=black] (4,-4) circle (2pt);
    \draw[fill=black] (6,-4) circle (2pt);
    \draw[fill=black] (8,-4) circle (2pt);
    \draw[fill=black] (10,-4) circle (2pt);
    \draw[thick, OI1] (0,0) -- (0,-4);
    \draw[thick, OI1] (2,0) -- (2,-4);
    \draw[thick, OI2] (4,0) -- (4,-4);
    \draw[thick, OI2] (6,0) -- (6,-4);
    \draw[thick, OI2] (8,0) -- (8,-4);
    \draw[thick, OI3] (0,0) -- (10,-4);
    \draw[thick, OI3] (2,0) -- (10,-4);
    \draw[thick, OI3] (4,0) -- (10,-4);
    \draw[thick, OI3] (6,0) -- (10,-4);
    \draw[thick, OI3] (8,0) -- (10,-4);
    \node at (-1,0) {$C_{+}$};
    \node at (-1,-4) {$C_{-}$};
    \node at (-1,-2) {$C$};
  \end{tikzpicture}
  \qquad\qquad
  \begin{tikzpicture}[scale=.5]
    \draw[fill=black] (0,0) circle (2pt);
    \draw[fill=black] (2,0) circle (2pt);
    \draw[fill=black] (0,-4) circle (2pt);
    \draw[fill=black] (2,-4) circle (2pt);
    \draw[thick, OI1] (0,0) -- (0,-4);
    \draw[thick, OI1] (2,0) -- (2,-4);
    \node at (-1,0) {$B_{+}$};
    \node at (-1,-4) {$B_{-}$};
    \node at (-1,-2) {$B$};
  \end{tikzpicture}
  \qquad\qquad
  \begin{tikzpicture}[scale=.5]
    \draw[fill=black] (0,0) circle (2pt);
    \draw[fill=black] (2,0) circle (2pt);
    \draw[fill=black] (4,0) circle (2pt);
    \draw[fill=black] (0,-4) circle (2pt);
    \draw[fill=black] (2,-4) circle (2pt);
    \draw[fill=black] (4,-4) circle (2pt);
    \draw[thick, OI2] (0,0) -- (0,-4);
    \draw[thick, OI2] (2,0) -- (2,-4);
    \draw[thick, OI2] (4,0) -- (4,-4);
    \node at (-1,0) {$B'_{+}$};
    \node at (-1,-4) {$B'_{-}$};
    \node at (-1,-2) {$B'$};
  \end{tikzpicture}
  \caption{There relations $\bm{C}$ and $\bm{B}\oplus\bm{B'}$ are bimorphic, but $\bm{C}$ is not the disjoint union of nonempty subrelations.\label{fig:morphism vs union}}
\end{figure}

\section{The skeleton of a relation}
\label{sec:lemmas}

In this section, we define the skeleton of a finite binary relation $\bm{A}$. This will be a binary relation which is obtainable from $\bm{A}$ by deleting points, which is bimorphic to $\bm{A}$, and which is minimal with respect to certain deletion operations.


\begin{definition}
  Let $\bm{A}$ be a binary relation, and $a\in A_-\cup A_+$. The \emph{neighborhood} of $a$, denoted $N_{\bm{A}}(a)$, is defined to be the set of all $b\in A_-\cup A_+$ such that $a\mathrel{A}b$ or $b\mathrel{A}a$.
\end{definition}

Observe that if $a\in A_-$, then $N_{\bm{A}}(a)\subseteq A_+$, and vice versa.

\begin{definition}
  Let $\bm{A}$ be a binary relation, and $a,b\in A_-\cup A_+$.
  \begin{enumerate}
    \item We say $a$ is \emph{non-maximal} if there exists $c\in A_-\cup A_+$ such that $N_{\bm{A}}(a)\subsetneq N_{\bm{A}}(c)$. 
    \item We say $a$ is \emph{non-minimal} if there exists $c\in A_-\cup A_+$ so that $N_{\bm{A}}(c)\subsetneq N_{\bm{A}}(a)$.
    \item We say $a$ and $b$ are \emph{twins} if $N_{\bm{A}}(a)=N_{\bm{A}}(b)$.
  \end{enumerate}
\end{definition}

\begin{lemma}
  Suppose $\bm{A'}$ is obtained from $\bm{A}$ by deleting either
  \begin{enumerate}
    \item a non-minimal point from $A_-$, or
    \item a non-maximal point from $A_+$, or
    \item any point from $A_-\cup A_+$ which has a twin.
  \end{enumerate}
  Then $\bm{A}$ and $\bm{A'}$ are bimorphic.
\end{lemma}

\begin{proof}
  Suppose $c\in A_-$ is non-minimal or has a twin, and that $\bm{A'}$ is obtained from $\bm{A}$ by deleting $c$.  By Lemma \ref{lemma:identity morphisms}(b), $\bm{A}\to\bm{A'}$.  

  Next, choose $c'\in A_-$, $c'\neq c$, so that $N_{\bm{A}}(c')\subseteq N_{\bm{A}}(c)$.  Define $\phi_-:A_-\to A'_-$ by $\phi_-(c)=c'$, and $\phi_-(x)=x$ when $x\neq c$.  Then $(\phi_-,\id_{A_+})$ is a morphism from $\bm{A'}$ to $\bm{A}$.

  The proof of the remaining case is similar.  We note it may also be established using duality, since a point in $A_-$ is non-minimal if and only if it is non-maximal in $A^{\perp}_+$.
\end{proof}

Given a relation $\bm{A}$, let us call a point \emph{deletable} if it is a non-minimal point from $A_-$ or a non-maximal point from $A_+$.  Deleting these points may create new deleteable points.  However, the next lemma will enable us to describe a coherent deletion algorithm.

\begin{lemma}
  Suppose $\bm{A'}$ is obtained from $\bm{A}$ by deleting a deletable point $a$. If $b\neq a$ is deletable in $\bm{A}$, then $b$ either remains deletable in $\bm{A'}$, or has a twin in $\bm{A'}$.
\end{lemma}

\begin{proof}
  We will show the case where $b \in A_{-}$ is non-minimal.  Then there exists $z \in A_{-}$ such that $N_\mathbf{A}(z) \subsetneq N_\mathbf{A}(b)$. Without loss of generality, suppose $z$ is $\mathbf{A}$-minimal.

  In the first case, if $a\in A_{-}$, or if $a\in A_{+}$ and $b\mathrel{\mathord{\neg}A}a$, then $z\mathrel{\mathord{\neg}A}a$. Hence $N_{\mathbf{A}^{'}}(b) = N_\mathbf{A}(b)$ and $N_{\mathbf{A}^{'}}(z) = N_\mathbf{A}(z)$, so $b$ remains non-minimal in $\mathbf{A}^{'}$

  Next suppose that $b\mathrel{A}a$ and $z\mathrel{A}a$. Then $N_{\mathbf{A}^{'}}(z) \subsetneq N_{\mathbf{A}^{'}}(b)$ and $b$ remains non-minimal in $\mathbf{A}^{'}$. 
        
  Finally suppose that $b\mathrel{A}a$ and $z\mathrel{\mathord{\neg}A}a$. Then since $z$ is $\mathbf{A}$-minimal, $N_{\mathbf{A}^{'}}(z) \subseteq N_{\mathbf{A}^{'}}(b)$ and either $b$ is non-minimal in $\mathbf{A}^{'}$ or $N_{\mathbf{A}^{'}}(z) = N_{\mathbf{A}^{'}}(b)$ and $b$ is a twin of $z$ in  $\mathbf{A}^{'}$.
\end{proof}

\begin{definition}
  \label{definition: skeleton}
  Let $\bm{A}$ be a finite binary relation. The \emph{skeleton bimorphic form} of $\bm{A}$ is the relation obtained from the following procedure:
  \begin{enumerate}
    \item Delete all non-maximal points from $A_+$ and all non-minimal points from $A_-$, if there are any;
    \item Repeat the previous step until there are no such points to delete;
    \item Delete all but one point from each set of twins.
  \end{enumerate}
  We call a relation \emph{skeletal} if it is equal to its skeleton bimorphic form.
\end{definition} 

Refer to Figure~\ref{fig:skeleton} for a sample run of the skeleton procedure.

\begin{figure}[!htb]
  \centering
  \begin{tikzpicture}[scale=.5,baseline=-1cm]
    \draw[fill=black] (0,0) circle (2pt);
    \draw[fill=black] (2,0) circle (2pt);
    \draw[fill=black] (4,0) circle (2pt);
    \draw[fill=black] (6,0) circle (2pt);
    \draw[fill=black] (0,-4) circle (2pt);
    \draw[fill=black] (2,-4) circle (2pt);
    \draw[fill=black] (4,-4) circle (2pt);
    \draw[fill=black] (6,-4) circle (2pt);
    \draw[fill=black] (8,-4) circle (2pt);
    \draw[thick, red] (2,0) circle (10pt);
    \draw[thick, red] (4,0) circle (10pt);
    \draw[thick, red] (6,0) circle (10pt);
    \draw[thick, blue] (0,-4) circle (10pt);
    \draw[thick, blue] (2,-4) circle (10pt);
    \draw[thick, blue] (6,-4) circle (10pt);
    \draw[thick, blue] (8,-4) circle (10pt);
    \draw[thick] (0,0) -- (0,-4);
    \draw[thick] (2,0) -- (0,-4);
    \draw[thick] (2,0) -- (2,-4);
    \draw[thick] (4,0) -- (2,-4);  
    \draw[thick] (4,0) -- (4,-4); 
    \draw[thick] (6,0) -- (4,-4);  
    \draw[thick] (6,0) -- (6,-4); 
    \draw[thick] (6,0) -- (8,-4); 
  \end{tikzpicture}
  \qquad$\implies$\qquad
  \begin{tikzpicture}[scale=.5,baseline=-1cm]
    \draw[fill=black] (2,0) circle (2pt);
    \draw[fill=black] (4,0) circle (2pt);
    \draw[fill=black] (6,0) circle (2pt);
    \draw[fill=black] (0,-4) circle (2pt);
    \draw[fill=black] (2,-4) circle (2pt);
    \draw[fill=black] (6,-4) circle (2pt);
    \draw[fill=black] (8,-4) circle (2pt);
    \draw[thick, red] (2,0) circle (10pt);
    \draw[thick, red] (6,0) circle (10pt);
    \draw[thick, blue] (0,-4) circle (10pt);
    \draw[thick, blue] (6,-4) circle (10pt);
    \draw[thick, blue] (8,-4) circle (10pt);
    \draw[thick] (2,0) -- (0,-4);
    \draw[thick] (2,0) -- (2,-4);
    \draw[thick] (4,0) -- (2,-4);  
    \draw[thick] (6,0) -- (6,-4); 
    \draw[thick] (6,0) -- (8,-4); 
  \end{tikzpicture}
  
  \vspace{10mm}
  $\implies$\qquad
  \begin{tikzpicture}[scale=.5,baseline=-1cm]
    \draw[fill=black] (2,0) circle (2pt);
    \draw[fill=black] (6,0) circle (2pt);
    \draw[fill=black] (0,-4) circle (2pt);
    \draw[fill=black] (6,-4) circle (2pt);
    \draw[fill=black] (8,-4) circle (2pt);
    \draw[thick, red] (2,0) circle (10pt);
    \draw[thick, red] (6,0) circle (10pt);
    \draw[thick, blue] (0,-4) circle (10pt);
    \draw[thick, blue] (6,-4) circle (10pt);
    \draw[thick, blue] (8,-4) circle (10pt);
    \draw[thick] (2,0) -- (0,-4);
    \draw[thick] (6,0) -- (6,-4); 
    \draw[thick] (6,0) -- (8,-4); 
  \end{tikzpicture}    
  \qquad$\implies$\qquad
  \begin{tikzpicture}[scale=.5,baseline=-1cm]
    \draw[fill=black] (2,0) circle (2pt);
    \draw[fill=black] (6,0) circle (2pt);
    \draw[fill=black] (0,-4) circle (2pt);
    \draw[fill=black] (6,-4) circle (2pt);
    \draw[fill=white] (8,-4) circle (0pt);
    \draw[thick] (2,0) -- (0,-4);
    \draw[thick] (6,0) -- (6,-4); 
  \end{tikzpicture}
  \caption{Finding the Skeleton Bimorphic Form of a Relation. Maximal points in $A_{+}$ are circled in red. Minimal points in $A_{-}$ are circled in blue.\label{fig:skeleton}}
\end{figure}

It is not clear to us whether we needed to delete the deletable points before deleting the extra twins:

\begin{question}
  Is there a finite binary relation, together with a sequence of deletions of twins and deletable points, which results in a skeletal relation which is non-isomorphic to the skeleton produced by our algorithm?
\end{question}

We have not considered any analog of skeleton for infinite binary relations:

\begin{question}
  Does the skeleton algorithm terminate for infinite binary relations?
\end{question}

\section{Classification of small binary relations}
\label{sec:small}

In this section we discuss the classification of all binary relations with at most $6$ vertices in $A_\pm$. To facilitate this discussion, let's say that the \emph{order} of $\bm{A}$ is the value of $\max(|A_-|,|A_+|)$. We wrote a computer program to calculate the skeleton of every binary relation $\bm{A}$ of order at most $6$, and to further calculate which pairs have a morphism between them.

The first part of the program makes an exhaustive search of all binary relations with $|A_\pm|=6$, in each case applying the algorithm in Definition~\ref{definition: skeleton} to calculate the skeleton of the relation. The search included all skeletons of binary relations of order $\leq6$, because the operation of adding a point which is not adjacent to any of the others produces a bimorphic relation. The program found $394$ non-isomorphic skeletal binary relations.

The second part of the program checked for all possible morphisms between the $394$ skeletal binary relations. Here in order for the program to finish it was necessary to use several of the lemmas from previous sections, which are shortcuts in the search for a morphism. Specifically, the program is able to use each of the following from this paper: \ref{prop:dual dominating number}, \ref{cor:dual dom number}, \ref{prop: one infinity}, \ref{lemma-ladders dominate}, \ref{cor: ladders}, \ref{cor:dual bimorphic}, \ref{lem:surjhomomorphism}, \ref{lem-disjoint union converse}.

Additionally, the program uses the following simple observation which does not appear elsewhere in this paper: There exists a morphism $\bm{A}\to\bm{B}$ if and only if there exists a morphism $\bm{A'}\to\bm{B}$ where $\bm{A'}$ is constructed from $\bm{A}$ by deleting vertices from $A_-$ until its size is equal to that of $B_-$. The proof may be found in \cite[Lemma~3.4.8]{barton}. Finally, the program makes significant use of the transitiviy of morphism existence.

The program was coded in the R language. The code and full output are both avilable at \href{https://github.com/rhettbarton/tukey-morphisms}{github.com/rhettbarton/tukey-morphisms}.

Focusing just on the binary relations of order $\leq5$, there are $32$ non-isomorphic skeletal binary relations. In Table~\ref{table:Skeletons} we list these $32$ skeletons along with a graph depicting each.

\begin{table}\small
  \begin{center}
  \begin{tabular}{c|c|c|c}
  \begin{tikzpicture}[scale=0.35]
    \node[circle,fill=black,inner sep=1pt] (a) at (0,0) {};
    \node[circle,fill=black,inner sep=1pt] (b) at (0,-4) {};
  \end{tikzpicture}
  &
  \begin{tikzpicture}[scale=0.35]
    \node[circle,fill=black,inner sep=1pt] (a) at (0,0) {};
    \node[circle,fill=black,inner sep=1pt] (b) at (0,-4) {};
    \draw (a) edge (b);
  \end{tikzpicture}
  &
  \begin{tikzpicture}[scale=0.35]
    \node[circle,fill=black,inner sep=1pt] (a1) at (0,0) {};
    \node[circle,fill=black,inner sep=1pt] (a2) at (2,0) {};
    \node[circle,fill=black,inner sep=1pt] (b1) at (0,-4) {};
    \node[circle,fill=black,inner sep=1pt] (b2) at (2,-4) {};
    \draw (a1) edge (b1)
          (a2) edge (b2);
  \end{tikzpicture}
  &
  \begin{tikzpicture}[scale=0.35]
    \node[circle,fill=black,inner sep=1pt] (a1) at (0,0) {};
    \node[circle,fill=black,inner sep=1pt] (a2) at (2,0) {};
    \node[circle,fill=black,inner sep=1pt] (a3) at (4,0) {};
    \node[circle,fill=black,inner sep=1pt] (b1) at (0,-4) {};
    \node[circle,fill=black,inner sep=1pt] (b2) at (2,-4) {};
    \node[circle,fill=black,inner sep=1pt] (b3) at (4,-4) {};
    \draw (a1) edge (b1)
          (a2) edge (b2)
          (a3) edge (b3);
  \end{tikzpicture}
  \\ 1&2&3&4
  \\[1ex]\hline&&&\\[1ex]
  \begin{tikzpicture}[scale=0.35]
    \node[circle,fill=black,inner sep=1pt] (a1) at (0,0) {};
    \node[circle,fill=black,inner sep=1pt] (a2) at (2,0) {};
    \node[circle,fill=black,inner sep=1pt] (a3) at (4,0) {};
    \node[circle,fill=black,inner sep=1pt] (b1) at (0,-4) {};
    \node[circle,fill=black,inner sep=1pt] (b2) at (2,-4) {};
    \node[circle,fill=black,inner sep=1pt] (b3) at (4,-4) {};
    \draw (a1) edge (b2) edge (b3)
          (a2) edge (b1) edge (b3)
          (a3) edge (b1) edge (b2);
  \end{tikzpicture}
  &
  \begin{tikzpicture}[scale=0.35]
    \node[circle,fill=black,inner sep=1pt] (a1) at (0,0) {};
    \node[circle,fill=black,inner sep=1pt] (a2) at (2,0) {};
    \node[circle,fill=black,inner sep=1pt] (a3) at (4,0) {};
    \node[circle,fill=black,inner sep=1pt] (a4) at (6,0) {};
    \node[circle,fill=black,inner sep=1pt] (b1) at (0,-4) {};
    \node[circle,fill=black,inner sep=1pt] (b2) at (2,-4) {};
    \node[circle,fill=black,inner sep=1pt] (b3) at (4,-4) {};
    \node[circle,fill=black,inner sep=1pt] (b4) at (6,-4) {};
    \draw (a1) edge (b1)
          (a2) edge (b2)
          (a3) edge (b3)
          (a4) edge (b4);
  \end{tikzpicture}
  &
  \begin{tikzpicture}[scale=0.35]
    \node[circle,fill=black,inner sep=1pt] (a1) at (0,0) {};
    \node[circle,fill=black,inner sep=1pt] (a2) at (2,0) {};
    \node[circle,fill=black,inner sep=1pt] (a3) at (4,0) {};
    \node[circle,fill=black,inner sep=1pt] (a4) at (6,0) {};
    \node[circle,fill=black,inner sep=1pt] (b1) at (0,-4) {};
    \node[circle,fill=black,inner sep=1pt] (b2) at (2,-4) {};
    \node[circle,fill=black,inner sep=1pt] (b3) at (4,-4) {};
    \node[circle,fill=black,inner sep=1pt] (b4) at (6,-4) {};
    \draw (a1) edge (b2) edge (b3)
          (a2) edge (b1) edge (b3)
          (a3) edge (b1) edge (b2)
          (a4) edge (b4);
  \end{tikzpicture}
  &
  \begin{tikzpicture}[scale=0.35]
    \node[circle,fill=black,inner sep=1pt] (a1) at (0,0) {};
    \node[circle,fill=black,inner sep=1pt] (a2) at (2,0) {};
    \node[circle,fill=black,inner sep=1pt] (a3) at (4,0) {};
    \node[circle,fill=black,inner sep=1pt] (a4) at (6,0) {};
    \node[circle,fill=black,inner sep=1pt] (b1) at (0,-4) {};
    \node[circle,fill=black,inner sep=1pt] (b2) at (2,-4) {};
    \node[circle,fill=black,inner sep=1pt] (b3) at (4,-4) {};
    \node[circle,fill=black,inner sep=1pt] (b4) at (6,-4) {};
    \draw (a1) edge (b2) edge (b3) edge (b4)
          (a2) edge (b1) edge (b4)
          (a3) edge (b1) edge (b3)
          (a4) edge (b1) edge (b2);
  \end{tikzpicture}
  \\ 5&6&7&8
  \\[1ex]\hline&&&\\[1ex]
  \begin{tikzpicture}[scale=0.35]
    \node[circle,fill=black,inner sep=1pt] (a1) at (0,0) {};
    \node[circle,fill=black,inner sep=1pt] (a2) at (2,0) {};
    \node[circle,fill=black,inner sep=1pt] (a3) at (4,0) {};
    \node[circle,fill=black,inner sep=1pt] (a4) at (6,0) {};
    \node[circle,fill=black,inner sep=1pt] (b1) at (0,-4) {};
    \node[circle,fill=black,inner sep=1pt] (b2) at (2,-4) {};
    \node[circle,fill=black,inner sep=1pt] (b3) at (4,-4) {};
    \node[circle,fill=black,inner sep=1pt] (b4) at (6,-4) {};
    \draw (a1) edge (b3) edge (b4)
          (a2) edge (b2) edge (b4)
          (a3) edge (b1) edge (b3)
          (a4) edge (b1) edge (b2);
  \end{tikzpicture}
  &
  \begin{tikzpicture}[scale=0.35]
    \node[circle,fill=black,inner sep=1pt] (a1) at (0,0) {};
    \node[circle,fill=black,inner sep=1pt] (a2) at (2,0) {};
    \node[circle,fill=black,inner sep=1pt] (a3) at (4,0) {};
    \node[circle,fill=black,inner sep=1pt] (a4) at (6,0) {};
    \node[circle,fill=black,inner sep=1pt] (b1) at (0,-4) {};
    \node[circle,fill=black,inner sep=1pt] (b2) at (2,-4) {};
    \node[circle,fill=black,inner sep=1pt] (b3) at (4,-4) {};
    \node[circle,fill=black,inner sep=1pt] (b4) at (6,-4) {};
    \draw (a1) edge (b2) edge (b3) edge (b4)
          (a2) edge (b1) edge (b3) edge (b4)
          (a3) edge (b1) edge (b2) edge (b4)
          (a4) edge (b1) edge (b2) edge (b3);
  \end{tikzpicture}
  &
  \begin{tikzpicture}[scale=0.35]
    \node[circle,fill=black,inner sep=1pt] (a1) at (0,0) {};
    \node[circle,fill=black,inner sep=1pt] (a2) at (2,0) {};
    \node[circle,fill=black,inner sep=1pt] (a3) at (4,0) {};
    \node[circle,fill=black,inner sep=1pt] (a4) at (6,0) {};
    \node[circle,fill=black,inner sep=1pt] (b1) at (0,-4) {};
    \node[circle,fill=black,inner sep=1pt] (b2) at (2,-4) {};
    \node[circle,fill=black,inner sep=1pt] (b3) at (4,-4) {};
    \node[circle,fill=black,inner sep=1pt] (b4) at (6,-4) {};
    \node[circle,fill=black,inner sep=1pt] (b5) at (8,-4) {};
    \draw (a1) edge (b1) edge (b3) edge (b5)
          (a2) edge (b1) edge (b2) edge (b4)
          (a3) edge (b4) edge (b5)
          (a4) edge (b2) edge (b3);
  \end{tikzpicture}
  &
  \begin{tikzpicture}[scale=0.35]
    \node[circle,fill=black,inner sep=1pt] (a1) at (0,0) {};
    \node[circle,fill=black,inner sep=1pt] (a2) at (2,0) {};
    \node[circle,fill=black,inner sep=1pt] (a3) at (4,0) {};
    \node[circle,fill=black,inner sep=1pt] (a4) at (6,0) {};
    \node[circle,fill=black,inner sep=1pt] (a5) at (8,0) {};
    \node[circle,fill=black,inner sep=1pt] (b1) at (0,-4) {};
    \node[circle,fill=black,inner sep=1pt] (b2) at (2,-4) {};
    \node[circle,fill=black,inner sep=1pt] (b3) at (4,-4) {};
    \node[circle,fill=black,inner sep=1pt] (b4) at (6,-4) {};
    \draw (a1) edge (b1) edge (b2)
          (a2) edge (b2) edge (b4)
          (a3) edge (b1) edge (b4)
          (a4) edge (b2) edge (b3)
          (a5) edge (b1) edge (b3);
  \end{tikzpicture}
  \\ 9&10&11&12
  \\[1ex]\hline&&&\\[1ex]
  \begin{tikzpicture}[scale=0.35]
    \node[circle,fill=black,inner sep=1pt] (a1) at (0,0) {};
    \node[circle,fill=black,inner sep=1pt] (a2) at (2,0) {};
    \node[circle,fill=black,inner sep=1pt] (a3) at (4,0) {};
    \node[circle,fill=black,inner sep=1pt] (a4) at (6,0) {};
    \node[circle,fill=black,inner sep=1pt] (a5) at (8,0) {};
    \node[circle,fill=black,inner sep=1pt] (b1) at (0,-4) {};
    \node[circle,fill=black,inner sep=1pt] (b2) at (2,-4) {};
    \node[circle,fill=black,inner sep=1pt] (b3) at (4,-4) {};
    \node[circle,fill=black,inner sep=1pt] (b4) at (6,-4) {};
    \node[circle,fill=black,inner sep=1pt] (b5) at (8,-4) {};
    \draw (a1) edge (b1)
          (a2) edge (b2)
          (a3) edge (b3)
          (a4) edge (b4)
          (a5) edge (b5);
  \end{tikzpicture}
  &
  \begin{tikzpicture}[scale=0.35]
    \node[circle,fill=black,inner sep=1pt] (a1) at (0,0) {};
    \node[circle,fill=black,inner sep=1pt] (a2) at (2,0) {};
    \node[circle,fill=black,inner sep=1pt] (a3) at (4,0) {};
    \node[circle,fill=black,inner sep=1pt] (a4) at (6,0) {};
    \node[circle,fill=black,inner sep=1pt] (a5) at (8,0) {};
    \node[circle,fill=black,inner sep=1pt] (b1) at (0,-4) {};
    \node[circle,fill=black,inner sep=1pt] (b2) at (2,-4) {};
    \node[circle,fill=black,inner sep=1pt] (b3) at (4,-4) {};
    \node[circle,fill=black,inner sep=1pt] (b4) at (6,-4) {};
    \node[circle,fill=black,inner sep=1pt] (b5) at (8,-4) {};
    \draw (a1) edge (b2) edge (b3)
          (a2) edge (b1) edge (b3)
          (a3) edge (b1) edge (b2)
          (a4) edge (b5)
          (a5) edge (b4);
  \end{tikzpicture}
  &
  \begin{tikzpicture}[scale=0.35]
    \node[circle,fill=black,inner sep=1pt] (a1) at (0,0) {};
    \node[circle,fill=black,inner sep=1pt] (a2) at (2,0) {};
    \node[circle,fill=black,inner sep=1pt] (a3) at (4,0) {};
    \node[circle,fill=black,inner sep=1pt] (a4) at (6,0) {};
    \node[circle,fill=black,inner sep=1pt] (a5) at (8,0) {};
    \node[circle,fill=black,inner sep=1pt] (b1) at (0,-4) {};
    \node[circle,fill=black,inner sep=1pt] (b2) at (2,-4) {};
    \node[circle,fill=black,inner sep=1pt] (b3) at (4,-4) {};
    \node[circle,fill=black,inner sep=1pt] (b4) at (6,-4) {};
    \node[circle,fill=black,inner sep=1pt] (b5) at (8,-4) {};
    \draw (a1) edge (b2) edge (b3) edge (b4)
          (a2) edge (b1) edge (b4)
          (a3) edge (b1) edge (b3)
          (a4) edge (b1) edge (b2)
          (a5) edge (b5);
  \end{tikzpicture}
  &
  \begin{tikzpicture}[scale=0.35]
    \node[circle,fill=black,inner sep=1pt] (a1) at (0,0) {};
    \node[circle,fill=black,inner sep=1pt] (a2) at (2,0) {};
    \node[circle,fill=black,inner sep=1pt] (a3) at (4,0) {};
    \node[circle,fill=black,inner sep=1pt] (a4) at (6,0) {};
    \node[circle,fill=black,inner sep=1pt] (a5) at (8,0) {};
    \node[circle,fill=black,inner sep=1pt] (b1) at (0,-4) {};
    \node[circle,fill=black,inner sep=1pt] (b2) at (2,-4) {};
    \node[circle,fill=black,inner sep=1pt] (b3) at (4,-4) {};
    \node[circle,fill=black,inner sep=1pt] (b4) at (6,-4) {};
    \node[circle,fill=black,inner sep=1pt] (b5) at (8,-4) {};
    \draw (a1) edge (b3) edge (b4)
          (a2) edge (b2) edge (b4)
          (a3) edge (b1) edge (b3)
          (a4) edge (b1) edge (b2)
          (a5) edge (b5);
  \end{tikzpicture}
  \\ 13&14&15&16
  \\[1ex]\hline&&&\\[1ex]
  \begin{tikzpicture}[scale=0.35]
    \node[circle,fill=black,inner sep=1pt] (a1) at (0,0) {};
    \node[circle,fill=black,inner sep=1pt] (a2) at (2,0) {};
    \node[circle,fill=black,inner sep=1pt] (a3) at (4,0) {};
    \node[circle,fill=black,inner sep=1pt] (a4) at (6,0) {};
    \node[circle,fill=black,inner sep=1pt] (a5) at (8,0) {};
    \node[circle,fill=black,inner sep=1pt] (b1) at (0,-4) {};
    \node[circle,fill=black,inner sep=1pt] (b2) at (2,-4) {};
    \node[circle,fill=black,inner sep=1pt] (b3) at (4,-4) {};
    \node[circle,fill=black,inner sep=1pt] (b4) at (6,-4) {};
    \node[circle,fill=black,inner sep=1pt] (b5) at (8,-4) {};
    \draw (a1) edge (b2) edge (b3) edge (b4)
          (a2) edge (b1) edge (b3) edge (b4)
          (a3) edge (b1) edge (b2) edge (b3)
          (a4) edge (b1) edge (b2) edge (b3)
          (a5) edge (b5);
  \end{tikzpicture}
  &
  \begin{tikzpicture}[scale=0.35]
    \node[circle,fill=black,inner sep=1pt] (a1) at (0,0) {};
    \node[circle,fill=black,inner sep=1pt] (a2) at (2,0) {};
    \node[circle,fill=black,inner sep=1pt] (a3) at (4,0) {};
    \node[circle,fill=black,inner sep=1pt] (a4) at (6,0) {};
    \node[circle,fill=black,inner sep=1pt] (a5) at (8,0) {};
    \node[circle,fill=black,inner sep=1pt] (b1) at (0,-4) {};
    \node[circle,fill=black,inner sep=1pt] (b2) at (2,-4) {};
    \node[circle,fill=black,inner sep=1pt] (b3) at (4,-4) {};
    \node[circle,fill=black,inner sep=1pt] (b4) at (6,-4) {};
    \node[circle,fill=black,inner sep=1pt] (b5) at (8,-4) {};
    \draw (a1) edge (b2) edge (b3) edge (b4) edge (b5)
          (a2) edge (b1) edge (b5)
          (a3) edge (b1) edge (b4)
          (a4) edge (b1) edge (b3)
          (a5) edge (b1) edge (b2);
  \end{tikzpicture}
  &
  \begin{tikzpicture}[scale=0.35]
    \node[circle,fill=black,inner sep=1pt] (a1) at (0,0) {};
    \node[circle,fill=black,inner sep=1pt] (a2) at (2,0) {};
    \node[circle,fill=black,inner sep=1pt] (a3) at (4,0) {};
    \node[circle,fill=black,inner sep=1pt] (a4) at (6,0) {};
    \node[circle,fill=black,inner sep=1pt] (a5) at (8,0) {};
    \node[circle,fill=black,inner sep=1pt] (b1) at (0,-4) {};
    \node[circle,fill=black,inner sep=1pt] (b2) at (2,-4) {};
    \node[circle,fill=black,inner sep=1pt] (b3) at (4,-4) {};
    \node[circle,fill=black,inner sep=1pt] (b4) at (6,-4) {};
    \node[circle,fill=black,inner sep=1pt] (b5) at (8,-4) {};
    \draw (a1) edge (b2) edge (b3) edge (b4)
          (a2) edge (b1) edge (b5)
          (a3) edge (b1) edge (b4)
          (a4) edge (b1) edge (b3)
          (a5) edge (b2) edge (b4);
  \end{tikzpicture}
  &
  \begin{tikzpicture}[scale=0.35]
    \node[circle,fill=black,inner sep=1pt] (a1) at (0,0) {};
    \node[circle,fill=black,inner sep=1pt] (a2) at (2,0) {};
    \node[circle,fill=black,inner sep=1pt] (a3) at (4,0) {};
    \node[circle,fill=black,inner sep=1pt] (a4) at (6,0) {};
    \node[circle,fill=black,inner sep=1pt] (a5) at (8,0) {};
    \node[circle,fill=black,inner sep=1pt] (b1) at (0,-4) {};
    \node[circle,fill=black,inner sep=1pt] (b2) at (2,-4) {};
    \node[circle,fill=black,inner sep=1pt] (b3) at (4,-4) {};
    \node[circle,fill=black,inner sep=1pt] (b4) at (6,-4) {};
    \node[circle,fill=black,inner sep=1pt] (b5) at (8,-4) {};
    \draw (a1) edge (b4) edge (b5)
          (a2) edge (b3) edge (b5)
          (a3) edge (b2) edge (b4)
          (a4) edge (b1) edge (b3)
          (a5) edge (b1) edge (b2);
  \end{tikzpicture}
  \\ 17&18&19&20
  \\[1ex]\hline&&&\\[1ex]
  \begin{tikzpicture}[scale=0.35]
    \node[circle,fill=black,inner sep=1pt] (a1) at (0,0) {};
    \node[circle,fill=black,inner sep=1pt] (a2) at (2,0) {};
    \node[circle,fill=black,inner sep=1pt] (a3) at (4,0) {};
    \node[circle,fill=black,inner sep=1pt] (a4) at (6,0) {};
    \node[circle,fill=black,inner sep=1pt] (a5) at (8,0) {};
    \node[circle,fill=black,inner sep=1pt] (b1) at (0,-4) {};
    \node[circle,fill=black,inner sep=1pt] (b2) at (2,-4) {};
    \node[circle,fill=black,inner sep=1pt] (b3) at (4,-4) {};
    \node[circle,fill=black,inner sep=1pt] (b4) at (6,-4) {};
    \node[circle,fill=black,inner sep=1pt] (b5) at (8,-4) {};
    \draw (a1) edge (b1) edge (b2) edge (b3)
          (a2) edge (b1) edge (b5)
          (a3) edge (b1) edge (b4)
          (a4) edge (b3) edge (b5)
          (a5) edge (b2) edge (b4);
  \end{tikzpicture}
  &
  \begin{tikzpicture}[scale=0.35]
    \node[circle,fill=black,inner sep=1pt] (a1) at (0,0) {};
    \node[circle,fill=black,inner sep=1pt] (a2) at (2,0) {};
    \node[circle,fill=black,inner sep=1pt] (a3) at (4,0) {};
    \node[circle,fill=black,inner sep=1pt] (a4) at (6,0) {};
    \node[circle,fill=black,inner sep=1pt] (a5) at (8,0) {};
    \node[circle,fill=black,inner sep=1pt] (b1) at (0,-4) {};
    \node[circle,fill=black,inner sep=1pt] (b2) at (2,-4) {};
    \node[circle,fill=black,inner sep=1pt] (b3) at (4,-4) {};
    \node[circle,fill=black,inner sep=1pt] (b4) at (6,-4) {};
    \node[circle,fill=black,inner sep=1pt] (b5) at (8,-4) {};
    \draw (a1) edge (b1) edge (b3) edge (b5)
          (a2) edge (b2) edge (b3) edge (b4)
          (a3) edge (b1) edge (b2)
          (a4) edge (b2) edge (b5)
          (a5) edge (b1) edge (b4);
  \end{tikzpicture}
  &
  \begin{tikzpicture}[scale=0.35]
    \node[circle,fill=black,inner sep=1pt] (a1) at (0,0) {};
    \node[circle,fill=black,inner sep=1pt] (a2) at (2,0) {};
    \node[circle,fill=black,inner sep=1pt] (a3) at (4,0) {};
    \node[circle,fill=black,inner sep=1pt] (a4) at (6,0) {};
    \node[circle,fill=black,inner sep=1pt] (a5) at (8,0) {};
    \node[circle,fill=black,inner sep=1pt] (b1) at (0,-4) {};
    \node[circle,fill=black,inner sep=1pt] (b2) at (2,-4) {};
    \node[circle,fill=black,inner sep=1pt] (b3) at (4,-4) {};
    \node[circle,fill=black,inner sep=1pt] (b4) at (6,-4) {};
    \node[circle,fill=black,inner sep=1pt] (b5) at (8,-4) {};
    \draw (a1) edge (b1) edge (b2) edge (b4)
          (a2) edge (b1) edge (b2) edge (b3)
          (a3) edge (b2) edge (b5)
          (a4) edge (b1) edge (b5)
          (a5) edge (b3) edge (b4);
  \end{tikzpicture}
  &
  \begin{tikzpicture}[scale=0.35]
    \node[circle,fill=black,inner sep=1pt] (a1) at (0,0) {};
    \node[circle,fill=black,inner sep=1pt] (a2) at (2,0) {};
    \node[circle,fill=black,inner sep=1pt] (a3) at (4,0) {};
    \node[circle,fill=black,inner sep=1pt] (a4) at (6,0) {};
    \node[circle,fill=black,inner sep=1pt] (a5) at (8,0) {};
    \node[circle,fill=black,inner sep=1pt] (b1) at (0,-4) {};
    \node[circle,fill=black,inner sep=1pt] (b2) at (2,-4) {};
    \node[circle,fill=black,inner sep=1pt] (b3) at (4,-4) {};
    \node[circle,fill=black,inner sep=1pt] (b4) at (6,-4) {};
    \node[circle,fill=black,inner sep=1pt] (b5) at (8,-4) {};
    \draw (a1) edge (b1) edge (b2) edge (b5)
          (a2) edge (b1) edge (b2) edge (b4)
          (a3) edge (b3) edge (b4) edge (b5)
          (a4) edge (b2) edge (b3)
          (a5) edge (b1) edge (b3);
  \end{tikzpicture}
  \\ 21&22&23&24
  \\[1ex]\hline&&&\\[1ex]
  \begin{tikzpicture}[scale=0.35]
    \node[circle,fill=black,inner sep=1pt] (a1) at (0,0) {};
    \node[circle,fill=black,inner sep=1pt] (a2) at (2,0) {};
    \node[circle,fill=black,inner sep=1pt] (a3) at (4,0) {};
    \node[circle,fill=black,inner sep=1pt] (a4) at (6,0) {};
    \node[circle,fill=black,inner sep=1pt] (a5) at (8,0) {};
    \node[circle,fill=black,inner sep=1pt] (b1) at (0,-4) {};
    \node[circle,fill=black,inner sep=1pt] (b2) at (2,-4) {};
    \node[circle,fill=black,inner sep=1pt] (b3) at (4,-4) {};
    \node[circle,fill=black,inner sep=1pt] (b4) at (6,-4) {};
    \node[circle,fill=black,inner sep=1pt] (b5) at (8,-4) {};
    \draw (a1) edge (b1) edge (b2) edge (b3)
          (a2) edge (b1) edge (b2) edge (b5)
          (a3) edge (b1) edge (b3) edge (b4)
          (a4) edge (b3) edge (b5)
          (a5) edge (b2) edge (b4);
  \end{tikzpicture}
  &
  \begin{tikzpicture}[scale=0.35]
    \node[circle,fill=black,inner sep=1pt] (a1) at (0,0) {};
    \node[circle,fill=black,inner sep=1pt] (a2) at (2,0) {};
    \node[circle,fill=black,inner sep=1pt] (a3) at (4,0) {};
    \node[circle,fill=black,inner sep=1pt] (a4) at (6,0) {};
    \node[circle,fill=black,inner sep=1pt] (a5) at (8,0) {};
    \node[circle,fill=black,inner sep=1pt] (b1) at (0,-4) {};
    \node[circle,fill=black,inner sep=1pt] (b2) at (2,-4) {};
    \node[circle,fill=black,inner sep=1pt] (b3) at (4,-4) {};
    \node[circle,fill=black,inner sep=1pt] (b4) at (6,-4) {};
    \node[circle,fill=black,inner sep=1pt] (b5) at (8,-4) {};
    \draw (a1) edge (b1) edge (b2) edge (b4)
          (a2) edge (b1) edge (b2) edge (b3)
          (a3) edge (b2) edge (b4) edge (b5)
          (a4) edge (b1) edge (b3) edge (b5)
          (a5) edge (b3) edge (b4);
  \end{tikzpicture}
  &
  \begin{tikzpicture}[scale=0.35]
    \node[circle,fill=black,inner sep=1pt] (a1) at (0,0) {};
    \node[circle,fill=black,inner sep=1pt] (a2) at (2,0) {};
    \node[circle,fill=black,inner sep=1pt] (a3) at (4,0) {};
    \node[circle,fill=black,inner sep=1pt] (a4) at (6,0) {};
    \node[circle,fill=black,inner sep=1pt] (a5) at (8,0) {};
    \node[circle,fill=black,inner sep=1pt] (b1) at (0,-4) {};
    \node[circle,fill=black,inner sep=1pt] (b2) at (2,-4) {};
    \node[circle,fill=black,inner sep=1pt] (b3) at (4,-4) {};
    \node[circle,fill=black,inner sep=1pt] (b4) at (6,-4) {};
    \node[circle,fill=black,inner sep=1pt] (b5) at (8,-4) {};
    \draw (a1) edge (b2) edge (b3) edge (b4) edge (b5)
          (a2) edge (b1) edge (b3) edge (b4)
          (a3) edge (b1) edge (b2) edge (b4)
          (a4) edge (b1) edge (b2) edge (b3)
          (a5) edge (b1) edge (b5);
  \end{tikzpicture}
  &
  \begin{tikzpicture}[scale=0.35]
    \node[circle,fill=black,inner sep=1pt] (a1) at (0,0) {};
    \node[circle,fill=black,inner sep=1pt] (a2) at (2,0) {};
    \node[circle,fill=black,inner sep=1pt] (a3) at (4,0) {};
    \node[circle,fill=black,inner sep=1pt] (a4) at (6,0) {};
    \node[circle,fill=black,inner sep=1pt] (a5) at (8,0) {};
    \node[circle,fill=black,inner sep=1pt] (b1) at (0,-4) {};
    \node[circle,fill=black,inner sep=1pt] (b2) at (2,-4) {};
    \node[circle,fill=black,inner sep=1pt] (b3) at (4,-4) {};
    \node[circle,fill=black,inner sep=1pt] (b4) at (6,-4) {};
    \node[circle,fill=black,inner sep=1pt] (b5) at (8,-4) {};
    \draw (a1) edge (b2) edge (b3) edge (b4)
          (a2) edge (b1) edge (b3) edge (b4)
          (a3) edge (b1) edge (b2) edge (b3)
          (a4) edge (b1) edge (b2) edge (b5)
          (a5) edge (b4) edge (b5);
  \end{tikzpicture}
  \\ 25&26&27&28
  \\[1ex]\hline&&&\\[1ex]
  \begin{tikzpicture}[scale=0.35]
    \node[circle,fill=black,inner sep=1pt] (a1) at (0,0) {};
    \node[circle,fill=black,inner sep=1pt] (a2) at (2,0) {};
    \node[circle,fill=black,inner sep=1pt] (a3) at (4,0) {};
    \node[circle,fill=black,inner sep=1pt] (a4) at (6,0) {};
    \node[circle,fill=black,inner sep=1pt] (a5) at (8,0) {};
    \node[circle,fill=black,inner sep=1pt] (b1) at (0,-4) {};
    \node[circle,fill=black,inner sep=1pt] (b2) at (2,-4) {};
    \node[circle,fill=black,inner sep=1pt] (b3) at (4,-4) {};
    \node[circle,fill=black,inner sep=1pt] (b4) at (6,-4) {};
    \node[circle,fill=black,inner sep=1pt] (b5) at (8,-4) {};
    \draw (a1) edge (b2) edge (b3) edge (b4) edge (b5)
          (a2) edge (b1) edge (b3) edge (b4) edge (b5)
          (a3) edge (b1) edge (b2) edge (b5)
          (a4) edge (b1) edge (b2) edge (b4)
          (a5) edge (b1) edge (b2) edge (b3);
  \end{tikzpicture}
  &
  \begin{tikzpicture}[scale=0.35]
    \node[circle,fill=black,inner sep=1pt] (a1) at (0,0) {};
    \node[circle,fill=black,inner sep=1pt] (a2) at (2,0) {};
    \node[circle,fill=black,inner sep=1pt] (a3) at (4,0) {};
    \node[circle,fill=black,inner sep=1pt] (a4) at (6,0) {};
    \node[circle,fill=black,inner sep=1pt] (a5) at (8,0) {};
    \node[circle,fill=black,inner sep=1pt] (b1) at (0,-4) {};
    \node[circle,fill=black,inner sep=1pt] (b2) at (2,-4) {};
    \node[circle,fill=black,inner sep=1pt] (b3) at (4,-4) {};
    \node[circle,fill=black,inner sep=1pt] (b4) at (6,-4) {};
    \node[circle,fill=black,inner sep=1pt] (b5) at (8,-4) {};
    \draw (a1) edge (b2) edge (b3) edge (b4) edge (b5)
          (a2) edge (b1) edge (b4) edge (b5)
          (a3) edge (b1) edge (b3) edge (b5)
          (a4) edge (b1) edge (b2) edge (b4)
          (a5) edge (b1) edge (b2) edge (b3);
  \end{tikzpicture}
  &
  \begin{tikzpicture}[scale=0.35]
    \node[circle,fill=black,inner sep=1pt] (a1) at (0,0) {};
    \node[circle,fill=black,inner sep=1pt] (a2) at (2,0) {};
    \node[circle,fill=black,inner sep=1pt] (a3) at (4,0) {};
    \node[circle,fill=black,inner sep=1pt] (a4) at (6,0) {};
    \node[circle,fill=black,inner sep=1pt] (a5) at (8,0) {};
    \node[circle,fill=black,inner sep=1pt] (b1) at (0,-4) {};
    \node[circle,fill=black,inner sep=1pt] (b2) at (2,-4) {};
    \node[circle,fill=black,inner sep=1pt] (b3) at (4,-4) {};
    \node[circle,fill=black,inner sep=1pt] (b4) at (6,-4) {};
    \node[circle,fill=black,inner sep=1pt] (b5) at (8,-4) {};
    \draw (a1) edge (b3) edge (b4) edge (b5)
          (a2) edge (b2) edge (b4) edge (b5)
          (a3) edge (b1) edge (b3) edge (b5)
          (a4) edge (b1) edge (b2) edge (b4)
          (a5) edge (b1) edge (b2) edge (b3);
  \end{tikzpicture}
  &
  \begin{tikzpicture}[scale=0.35]
    \node[circle,fill=black,inner sep=1pt] (a1) at (0,0) {};
    \node[circle,fill=black,inner sep=1pt] (a2) at (2,0) {};
    \node[circle,fill=black,inner sep=1pt] (a3) at (4,0) {};
    \node[circle,fill=black,inner sep=1pt] (a4) at (6,0) {};
    \node[circle,fill=black,inner sep=1pt] (a5) at (8,0) {};
    \node[circle,fill=black,inner sep=1pt] (b1) at (0,-4) {};
    \node[circle,fill=black,inner sep=1pt] (b2) at (2,-4) {};
    \node[circle,fill=black,inner sep=1pt] (b3) at (4,-4) {};
    \node[circle,fill=black,inner sep=1pt] (b4) at (6,-4) {};
    \node[circle,fill=black,inner sep=1pt] (b5) at (8,-4) {};
    \draw (a1) edge (b2) edge (b3) edge (b4) edge (b5)
          (a2) edge (b1) edge (b3) edge (b4) edge (b5)
          (a3) edge (b1) edge (b2) edge (b4) edge (b5)
          (a4) edge (b1) edge (b2) edge (b3) edge (b5)
          (a5) edge (b1) edge (b2) edge (b3) edge (b4);
  \end{tikzpicture}
  \\ 29&30&31&32
  \end{tabular}
  \end{center}
  \caption{Skeleton forms of binary relations of order $\leq5$\label{table:Skeletons}}
\end{table}

In Figure~\ref{fig:hasse diagram}, we plot the Hasse diagram of the morphism quasi-order on the skeletal relations of order $\leq5$. In the diagram, the skeletal relations are also arranged according to their dominating number and dual dominating number.

\begin{figure}
  \centering
  \begin{tikzpicture}
    \node (a) at (0,1) {$\mathbf{A}$};
    \node (1) at (0,0) {$1$};
    \node (13) at (0,-1) {$13$};
    \node (6) at (0,-2) {$6$};
    \node (14) at (0,-3) {$14$};
    \node (4) at (0,-4) {$4,16$};
    \node (15) at (-2,-5) {$15$};
    \node (7) at (-2,-6) {$7$};
    \node (20) at (2,-5) {$20$};
    \node (21) at (0,-7) {$21$};
    \node (17) at (-2,-8) {$17$};
    \node (23) at (2,-8) {$23$};
    \node (3) at (0,-9) {$3,9,11,12,19,22,25,28$};
    \node (18) at (-2,-10) {$18$};
    \node (24) at (2,-10) {$24$};
    \node (26) at (0,-11) {$26$};
    \node (8) at (-2,-12) {$8$};
    \node (27) at (-2,-13) {$27$};
    \node (31) at (2,-13) {$31$};
    \node (5) at (0,-14) {$5,30$};
    \node (29) at (0,-15) {$29$};
    \node (10) at (0,-16) {$10$};
    \node (32) at (0,-17) {$32$};
    \node (2) at (0,-18) {$2$};
    \draw (1) -- (13) -- (6) -- (14) -- (4) -- (15) -- (7) --(17) -- (3);
    \draw (7) -- (21);
    \draw (4) -- (20) -- (21) -- (23) -- (3) -- (24) -- (26) -- (8);
    \draw (26) -- (31) -- (5);
    \draw (3) -- (18) -- (8) -- (27) -- (5) -- (29) -- (10) -- (32) -- (2);
    \node (d) at (-7,1) {$\delta(\mathbf{A})$};
    \node (dinf) at (-7,0) {$\infty$};
    \draw[OI1] (dinf) -- (1);
    \node (d5) at (-7,-1) {$5$};
    \draw[OI2] (d5) -- (13);
    \node (d4) at (-7,-2) {$4$};
    \draw[OI3] (d4) -- (6);
    \node (d3) at (-7,-4) {$3$};
    \draw[OI4] (d3) -- (4);
    \node (d2) at (-7,-9) {$2$};
    \draw[OI5] (d2) -- (3);
    \node (d1) at (-7,-18) {$1$};
    \draw[black] (d1) -- (2);
    
    \filldraw[OI3!20] (-6.75,-2) rectangle (-3.5,-3.5);
    \filldraw[OI4!20] (-6.75,-4) rectangle (-3.5,-8.5);
    \filldraw[OI5!20] (-6.75,-9) rectangle (-3.5,-17.5);
    
    \node (dd) at (7,1) {$\delta(\mathbf{A}^\bot)$};
    \node (dinf) at (7,-18) {$\infty$};
    \draw[OI1] (dinf) -- (2);
    \node (d5) at (7,-17) {$5$};
    \draw[OI2] (d5) -- (32);
    \node (d4) at (7,-16) {$4$};
    \draw[OI3] (d4) -- (10);
    \node (d3) at (7,-14) {$3$};
    \draw[OI4] (d3) -- (5);
    \node (d2) at (7,-9) {$2$};
    \draw[OI5] (d2) -- (3);
    \node (d1) at (7,0) {$1$};
    \draw[black] (d1) -- (1);
    
    \filldraw[OI3!20] (6.75,-16) rectangle (3.5,-14.5);
    \filldraw[OI4!20] (6.75,-14) rectangle (3.5,-9.5);
    \filldraw[OI5!20] (6.75,-9) rectangle (3.5,-0.5);           
  \end{tikzpicture}
  \caption{Hasse Diagram of skeleton forms of binary relations of order $\leq5$\label{fig:hasse diagram}}
\end{figure}

In the diagram, we can immediately observe that there are pairs of skeletal relations with identical dominating number and dual dominating number, but with no morphism between them, such as \#15 and \#20. Thus we found finite counterexamples to the converse of Corollary~\ref{cor:dual dom number}, as promised in the introduction.

The diagram helps visualize some of our results in Sections~\ref{sec:intro} and~\ref{sec:examples}. It also shows some surprising structural properties not reflected in our results, leading for instance to the following question.

\begin{question}
  If $\delta(\bm{A})\geq \delta(\bm{A}^\perp)$, then is it necessarily the case that $\bm{A}\to\bm{A}^\perp$?
\end{question}

Increasing our search parameter from order at most $5$ to order at most $6$ increased our computing time from a few minutes to $30$ hours.  This larger search uncovered our first binary relation $\bm{A}$ such that $\delta(\bm{A}),\delta(\bm{A}^\perp)\geq 3$, shown in Figure~\ref{fig:33}. The particular features of this example helped motivate the constructions in the next section.

\begin{figure}
  \centering
  \begin{tikzpicture}[scale=.5]
    \draw[fill=black] (0,0) circle (2pt);
    \draw[fill=black] (2,0) circle (2pt);
    \draw[fill=black] (4,0) circle (2pt);
    \draw[fill=black] (6,0) circle (2pt);
    \draw[fill=black] (8,0) circle (2pt);
    \draw[fill=black] (10,0) circle (2pt);
    \draw[fill=black] (0,-4) circle (2pt);
    \draw[fill=black] (2,-4) circle (2pt);
    \draw[fill=black] (4,-4) circle (2pt);
    \draw[fill=black] (6,-4) circle (2pt);
    \draw[fill=black] (8,-4) circle (2pt);
    \draw[fill=black] (10,-4) circle (2pt);
    \draw[thick, OI1] (0,0) -- (0,-4);
    \draw[thick, OI1] (2,0) -- (2,-4);
    \draw[thick, OI1] (0,0) -- (2,-4);
    \draw[thick, OI1] (2,0) -- (0,-4);
    \draw[thick, OI1] (4,0) -- (4,-4);
    \draw[thick, OI1] (6,0) -- (6,-4);
    \draw[thick, OI1] (4,0) -- (6,-4);
    \draw[thick, OI1] (6,0) -- (4,-4);
    \draw[thick, OI1] (8,0) -- (8,-4);
    \draw[thick, OI1] (10,0) -- (10,-4);
    \draw[thick, OI1] (8,0) -- (10,-4);
    \draw[thick, OI1] (10,0) -- (8,-4);
    \draw[thick, OI2] (4,0) -- (0,-4);
    \draw[thick, OI2] (6,0) -- (2,-4);
    \draw[thick, OI2] (8,0) -- (4,-4);
    \draw[thick, OI2] (10,0) -- (6,-4);
    \draw[thick, OI3] (0,0) -- (8,-4);
    \draw[thick, OI3] (2,0) -- (10,-4);
    \node at (-1,0) {$A_{+}$};
    \node at (-1,-4) {$A_{-}$};
    \node at (-1,-2) {$A$};
  \end{tikzpicture}
  \caption{An example of a binary relation $\bm{A}$ with $\delta(\bm{A})=\delta(\bm{A}^\perp)=3$.\label{fig:33}}
\end{figure}

\section{Structure of the Tukey morphism ordering}
\label{sec:construction}

In this section we show how to construct binary relations with given dominating number and dual dominating number.

To begin, we introduce the following matrix representation of binary relations. Let $\bm{A}$ be a binary relation, and let $A_-=\{x_1,\ldots,x_n\}$, $A_+=\{y_1,\ldots,y_m\}$. Then we say $\bm{A}$ is represented by the $n\times m$ matrix $(a_{ij})$ where $a_{ij}=1$ if $x_i\mathrel{A}y_j$ and $a_{ij}=0$ otherwise.

Observe that if we permute the (arbitrary) ordering of the elements of $A_\pm$, then the rows and columns of the representing matrix will be permuted accordingly. Conversely, permuting the rows and columns of a matrix representation of $\bm{A}$ results in a binary relation isomorphic to $\bm{A}$. Because of this, we will often conflate a binary relation $\bm{A}$ with a matrix representation of $\bm{A}$, and this should not cause confusion.

We note that the matrix representations of $\bm{A}$ and $\bm{A}^\perp$ are related in the following way: a representation of $\bm{A}$ is the transpose of the Boolean negation of a representation of $\bm{A}^\perp$.

For an example of a matrix representation, an $n$-ladder may be represented by the $n\times n$ identity matrix. Changing the order of the elements of $A_\pm$, an $n$-ladder may also be represented by any $n\times n$ permutation matrix. For another example, the binary relation in Figure~\ref{fig:33} is represented by the matrix shown in Figure~\ref{fig:33matrix}.

\begin{figure}[h!]
  \[\left(\begin{array}{cc|cc|cc}
    1&1&0&0&1&0\\
    1&1&0&0&0&1\\
    \hline
    1&0&1&1&0&0\\
    0&1&1&1&0&0\\
    \hline
    0&0&1&0&1&1\\
    0&0&0&1&1&1\\
    \end{array}\right)
  \]
  \caption{The matrix representation of the binary relation in Figure~\ref{fig:33}.\label{fig:33matrix}}
\end{figure}

The structure of this matrix helped motivate the proof of the following general construction.

\begin{theorem}
  \label{theorem:block1}
  For each $n\geq 2$, there is a binary relation $\bm{C_n}$ with $\lvert C_{n-}\rvert=\lvert C_{n+}\rvert=n!$ such that $\delta(\bm{C_n})=\delta(\bm{C_n^\perp})=n$.
\end{theorem}

\begin{proof}
  Let $\bm{C_2}$ be the $2$-ladder.  Given $\bm{C_n}$, we construct $\bm{C_{n+1}}$ so that it is isomorphic to its dual and $\delta(\bm{C_{n+1}})=n+1$.  This ensures that $\delta(\bm{C_{n+1}}^\perp)=n+1$ as well.  In fact, let $\bm{C_{n+1}}$ be the block matrix shown in Figure \ref{fig:block1}, where there are $n+1$ rows and columns of blocks, each of size $n!\times n!$.
  
  \begin{figure}[h!]
    \[\begin{pmatrix}
    \mathcal{J}_{n!} & \mathcal{O}_{n!} & \bm{C_n} & \bm{C_n} & \bm{C_n}  &\ldots & \bm{C_n}\\
    \bm{C_n} & \mathcal{J}_{n!} & \mathcal{O}_{n!} &  \bm{C_n} & \bm{C_n} &\ldots & \bm{C_n}\\
    \bm{C_n} & \bm{C_n} & \mathcal{J}_{n!} & \mathcal{O}_{n!} &  \bm{C_n} &\ldots & \bm{C_n}\\
    \vdots\\
    \vdots\\
    \bm{C_n} & \bm{C_n} & \bm{C_n}  &\ldots & \bm{C_n} & \mathcal{J}_{n!} & \mathcal{O}_{n!}\\
    \mathcal{O}_{n!} & \bm{C_n} & \bm{C_n} & \bm{C_n}  &\ldots & \bm{C_n} & \mathcal{J}_{n!}\\
    \end{pmatrix}
    \]
    \caption{The recursive construction in the proof of Theorem \ref{theorem:block1}.\label{fig:block1}}
  \end{figure} 
    
  Here $O_{n!}$ is the all-zeros matrix and $\mathcal{J}_{n!}$ is the all-ones matrix.  To see that $\bm{C_{n+1}}$ is isomorphic to its dual, first permute the row blocks and column blocks to interchange the $\mathcal{J}$s and $\mathcal{O}$s.  Then, within each block, apply the permutation of rows and columns which sends $\bm{C_n}$ to its dual.

  To show that $\delta(\bm{C_{n+1}})=n+1$, let $X\subseteq C_{(n+1)+}$ where $\lvert X \rvert\leq n$.   By pigeonhole, there is a $1\leq b\leq n+1$ so that the $b$th block column does not contain an element of $X$ and the $(b+1)$th $\pmod{n+1}$ block column does contain an element of $X$.
  This means that the elements of  $C_{(n+1)-}$ represented by the $b$th block row are covered by fewer than $n$ elements of identical copies of $C_{n+}$.  So by induction, not every element in this block row is covered by an element of $X$.
\end{proof}

Our data shows that this is optimal for $n=2,3$. That is, in order to obtain $\delta(\bm{A})=\delta(\bm{A^\perp})=2$ we need $|A_\pm|\geq2$, and to obtain $\delta(\bm{A})=\delta(\bm{A^\perp})=3$ we need $|A_\pm|\geq6$. We therefore ask the following.

\begin{question}
  Does there exist a binary relation $\bm{A}$ such that the order of $\bm{A}$ is less than $n!$, and $\delta(\bm{A})=\delta(\bm{A^\perp})=n$?
\end{question}

Our data furthermore shows that this construction is essentially unique for $n=2,3$.

\begin{question}
  Suppose that $\bm{A}$ is a binary relation of order at most $n!$, that $\delta(\bm{A})=\delta(\bm{A^\perp})=n$, and that $\bm{A}$ is isomorphic to $\bm{A^\perp}$. Does it follow that $\bm{A}$ is isomorphic to $\bm{C_n}$?
\end{question}

We next show how to obtain arbitrary dominating numbers and dual dominating numbers.

\begin{lemma}
  For each $2\leq k\leq n$, there is a relation $\bm{C_{nk}}=(C_{nk-},C_{nk+},C_{nk})$ so that $\delta(\bm{C_{nk}})=n$ and $\delta(\bm{C_{nk}^\perp})=k$.
\end{lemma}

\begin{proof}
  Fix $n\geq 2$ and let $C_{n2}$ be the $n$-ladder.  Given $k<n$ and $C_{nk}$, we let $C_{n(k+1)}$ be the following block matrix, where there are $n$ rows and columns of blocks, each of size $\lvert C_{nk-}\rvert \times \lvert C_{nk+}\rvert$.
  \begin{center}
    $\begin{pmatrix}
    \mathcal{J} &\mathcal{O} & \bm{C_{nk}} & \bm{C_{nk}} & \bm{C_{nk}}  &\ldots & \bm{C_{nk}}\\
    \bm{C_{nk}} & \mathcal{J} &\mathcal{O} &  \bm{C_{nk}} & \bm{C_{nk}} &\ldots & \bm{C_{nk}}\\
    \bm{C_{nk}} & \bm{C_{nk}} & \mathcal{J} &\mathcal{O} &  \bm{C_{nk}} &\ldots & \bm{C_{nk}}\\
    \vdots\\
    \vdots\\
    \bm{C_{nk}} & \bm{C_{nk}} & \bm{C_{nk}}  &\ldots & \bm{C_{nk}} & \mathcal{J} &\mathcal{O} \\
    O & \bm{C_{nk}} & \bm{C_{nk}} & \bm{C_{nk}}  &\ldots & \bm{C_{nk}} & \mathcal{J}\\
    \end{pmatrix}$
  \end{center}
  By taking one element from each block column, we obtain $\delta(\bm{C_{n(k+1)}})\leq n$.  Essentially the same argument from the previous proof then shows that $\delta(\bm{C_{n(k+1)}})\geq n$.

  Next consider $\bm{C_{n(k+1)}^\perp}$, which is isomorphic to the relation represented by the following block matrix.
  \begin{center}
    $\begin{pmatrix}
    \mathcal{J} &\mathcal{O} & \bm{C_{nk}^\perp} & \bm{C_{nk}^\perp} & \bm{C_{nk}^\perp}  &\ldots & \bm{C_{nk}^\perp}\\
    \bm{C_{nk}^\perp} & \mathcal{J} &\mathcal{O} &  \bm{C_{nk}^\perp} & \bm{C_{nk}^\perp} &\ldots & \bm{C_{nk}^\perp}\\
    \bm{C_{nk}^\perp} & \bm{C_{nk}^\perp} & \mathcal{J} &\mathcal{O} &  \bm{C_{nk}^\perp} &\ldots & \bm{C_{nk}^\perp}\\
    \vdots\\
    \vdots\\
    \bm{C_{nk}^\perp} & \bm{C_{nk}^\perp} & \bm{C_{nk}^\perp}  &\ldots & \bm{C_{nk}^\perp} & \mathcal{J} &\mathcal{O} \\
    \mathcal{O} & \bm{C_{nk}^\perp} & \bm{C_{nk}^\perp} & \bm{C_{nk}^\perp}  &\ldots & \bm{C_{nk}^\perp} & \mathcal{J}\\
    \end{pmatrix}$
  \end{center}
  To show $\delta(\bm{C_{n(k+1)}^\perp})\leq k+1$, consider a set of $k$ elements from the second block column which together cover all of $(\bm{C_{nk}^\perp})_-$ together with any element of the first block column.  In order to show $\delta(\bm{C_{n(k+1)}^\perp})\geq k+1$, let $X$ be a set of $k$ elements of $(\bm{C_{n(k+1)}^\perp})_+$.  By pigeonhole, there is a $1\leq b\leq n$ so that the $b$th block column does not contain an element of $X$ and the $(b+1)$th $\pmod{n}$ block column does contain an element of $X$.
  This means that the elements of  $(\bm{C_{n(k+1)}^\perp})_-$ represented by the $b$th block row are covered by fewer than $k$ elements of identical copies of $\bm{C_{nk}^\perp}$.
\end{proof}

We observe that the construction above produced very large binary relations.  Specifically, to obtain a $\delta(\bm{A})=n$ and $\delta(\bm{A}^\perp)=k$ with $k\leq n$, we needed $A_\pm$ to each have at least $n^{k-1}$ elements. In particular, if $n=k$, then this is much bigger than the $n!$ points we used in Lemma \ref{theorem:block1}. We therefore ask the following.

\begin{question}
  What is the smallest possible order of a relation $\bm{A}$ for which $\delta(\bm{A})=n$ and $\delta(\bm{A^\perp})=k$?
\end{question}

The construction above helps answer some of the most basic questions about the structure of the morphism quasi-order on the finite binary relations.

\begin{corollary}
  The morphism quasi-order on the finite binary relations contains infinite ascending chains, infinite descending chains, and infinite antichains.
\end{corollary}

\begin{question}
  What else can one show about the structure of the Tukey morphism quasi-order?
\end{question}

\bibliographystyle{plain}
\bibliography{tukeyfin}

\end{document}